\documentclass[11pt]{amsart}

\usepackage{latexsym,mathrsfs,bm}
\usepackage[english]{babel}
\usepackage{paralist}
\usepackage{ulem} 
\usepackage{amsfonts,amssymb,amsmath}
\usepackage{amsthm}
\usepackage[latin1]{inputenc}
\usepackage{bm}
\usepackage{color} 

\newcommand{\eps}{\varepsilon}
\newcommand{\R}{\mathbb{R}}
\newcommand{\Q}{\mathbb{Q}}
\newcommand{\C}{\mathbb{C}}
\newcommand{\N}{\mathbb{N}}

\newcommand{\Z}{\mathbb{Z}}

\newcommand{\es}[1]{\begin{equation}\begin{split}#1\end{split}\end{equation}}
\newcommand{\est}[1]{\begin{equation*}\begin{split}#1\end{split}\end{equation*}}

\renewcommand{\mod}[1]{~\pr{\textnormal{mod}~#1}}
\newtheorem*{theo*}{Theorem}
\newtheorem{theo}{Theorem}

\newtheorem{ezer}{Exercise}

\newtheorem{prop}[ezer]{Proposition}

\newtheorem{lemma}{Lemma}
\newtheorem{corol}[lemma]{Corollary}
\newtheorem{remark}{Remark}

\newtheorem*{rem*}{Remark}
\def\sumstar{\operatornamewithlimits{\sum\nolimits^*}}

\newcommand{\pro}[1]{\left( #1\right)}
\newcommand{\pr}[1]{\mathchoice{\left( #1\right)}{(#1)}{(#1)}{(#1)}}

\newcommand{\prbigg}[1]{\bigg( #1\bigg)}
\newcommand{\prBigg}[1]{\Bigg( #1\Bigg)}
\newcommand{\pg}[1]{\left\{ #1\right\}}

\newcommand{\e}[1]{\operatorname{e}\pr{ #1}}

\newcommand{\Gal}{\operatorname*{Gal}}

\newcommand{\Av}{\mathrm{Av}}
\newcommand{\Tr}{\mathrm{Tr}}
\newcommand{\comment}[1]{}

\let\originalleft\left
\let\originalright\right
\renewcommand{\left}{\mathopen{}\mathclose\bgroup\originalleft}
\renewcommand{\right}{\aftergroup\egroup\originalright}

\numberwithin{equation}{section}

\title{On the non-vanishing of certain Dirichlet series}

\author{Sandro Bettin}
\address{DIMA - Dipartimento di Matematica, Via Dodecaneso, 35, 16146 Genova - ITALY}
\email{bettin@dima.unige.it}

\author{Bruno Martin}
\address{
Unit\'e Mixte Internationale 3457 - CNRS \& CRM -  Universit\'e de Montr\'eal - CANADA.
} 
\email{Bruno.Martin@univ-littoral.fr}

\begin{document}

\begin{abstract}
Given $k\in\N$, we study the vanishing of the Dirichlet series 
\est{D_k(s,f):=\sum_{n\geq1} d_k(n)f(n)n^{-s}
}
at the point $s=1$, where $f$ is a periodic function modulo a prime $p$. We show that if $(k,p-1)=1$ or $(k,p-1)=2$ and $p\equiv 3\mod 4$, then there are no odd rational-valued functions $f\not\equiv 0$ such that $D_k(1,f)=0$, whereas in all other cases there are examples of odd functions $f$ such that $D_k(1,f)=0$.

As a consequence, we obtain, for example, that the set of values $L(1,\chi)^2$, where $\chi$ ranges over odd characters mod $p$, are linearly independent over $\Q$.

\end{abstract}

\subjclass[2010]{11M41, 11L03, 11M20 (primary), 11R18 (secondary)}

\keywords{Chowla's problem, non-vanishing of Dirichlet series, Special values of Dirichlet L-series}

\maketitle

\section{Introduction}
Let $p$ be prime and let $K$ be a number field. For a function $f:\Z\to K$ which is periodic modulo $p$, let $L(s,f)$ be the Dirichlet series
\est{
L(s,f):=\sum_{n= 1}^\infty\frac{f(n)}{n^s},
}
which is absolutely convergent for $\Re(s)>1$. Since $L(s,f)=p^{-s}\sum_{a=1}^pf(a)\zeta (s,a/p)$, where $\zeta(s,x)$ is the Hurwitz zeta-function which is meromorphic in $\C$ with a pole of residue $1$ at $s=1$ only, one has that $L(s,f)$ admits meromorphic continuation to $\C$ with (possibly) a simple pole at $s=1$ only of residue $ \Av(f)$ with 
\est{
\Av(f):=\frac1p\sum_{a\mod p}f(a).
}
In particular, if $\Av(f)=0$ then $L(s,f)$ is entire.

In the papers~\cite{Cho1,Cho2} Chowla asked whether it is possible that $L(1,f)=0$ for some rational-valued periodic function $f$ satisfying $\Av(f)=0$ and with $f$ not identically zero.
Following an approach outlined by Siegel, Chowla solved the problem in the case where $f$ is odd by showing that in this case $L(1,f)$ is never zero. Later, Baker, Birch and Wirsing~\cite{BBW} used Baker's theorem on linear forms in logarithms to give a complete answer to Chowla's question showing that $L(1,f)\neq0$ whenever $K\cap\Q(\xi_p)=\Q$, where $\xi_n:=\e{1/n}$ with $\e{x}:=e^{2\pi i x}$. In the following years Chowla's problem was considered and generalized by several other authors, for example we mention the work of Gun, Murty and Rath~\cite{GMR} where other points besides $s=1$ were considered (and where the condition on $K$ was slightly relaxed) and the works of Okada~\cite{Oka} and of Chatterjee and Murty~\cite{CM2014}, who gave equivalent criteria for the vanishing of $L(1,f)$ when no condition on $K$ is imposed. See also~\cite{MM} for a variation of the proof of the result by Baker, Birch and Wirsing.

In this paper we consider the analogue of Chowla's problem for
\est{
D_k(s,f):=\sum_{n=1}^\infty\frac{d_k(n)f(n)}{n^s}=\sum_{n_1,\dots,n_k=1}^\infty\frac{f(n_1\cdots n_k)}{(n_1\cdots n_k)^s},
}
where $d_k(n):=\sum_{m_1\cdots m_k=n}1$. As for $L(s,f)$, $D_k(s,f)$ is absolutely convergent for $\Re(s)>1$ and, expressing each of the series in the second expression for $D_k$ in terms of Hurwitz zeta-functions, one obtains analytic continuation for $D_k(s,f)$ to $\C\setminus\{1\}$. In the case where $k>1$, the analyticity of $D_k$ at $s=1$ is equivalent to having $\Av(f)=0$ and $f(0)=0$ (see Lemma~\ref{conv}). Notice that if $f$ is odd, then both conditions are automatically met.

If $f$ is not odd, then one can easily see that $D_k(1,f)\neq0$ by appealing to Schanuel's conjecture. We remind the reader that Schanuel's conjecture predicts that for any $z_1,\dots,z_n\in\C$ which are linearly independent over $\Q$ the transcendence degree of $\Q(z_1,...,z_n, e^{z_1},...,e^{z_n})$ over $\Q$ is at least $n$. 
\begin{prop}\label{sch}
Let $p\geq3$ be prime and let $k\in\N$. Let $f:\Z\to\overline\Q$ be $p$-periodic with $f(0)=\Av(f)=0.$ Then, under Schanuel's conjecture we have that if $D_{k}(1,f)=0$ then $f$ is odd.
\end{prop}
Proposition~\ref{sch} is an easy consequence of the fact that for $\chi$ odd $L(1,\chi)/\pi$ is an algebraic number whereas $\pi$ and the values $L(1,\chi)$, as $\chi$ ranges over even non-principal Dirichlet character mod $p$, are known to be algebraically independent under Schanuel's conjecture. In fact the full Schanuel's conjecture is not needed here, an analogue of Baker's theorem for linear forms in $k$-th powers of logarithms would suffice for Proposition~\ref{sch}.

Thus, at least conditionally, to determine whether $D(1,f)$ can be zero we just need to consider the case of $f$ odd. The case $(k,p-1)=1$ is completely analogous to the case $k=1$ and one has that $D_k(1,f)\neq0$ if $K\cap\Q(\xi_p)=\Q$. If $(k,p-1)>1$ then the situation changes drastically and already for $k=2$ and $p=5$ we can find non-trivial functions $f$ such that $D_2(1,f)=0$. Indeed, if $f$ is the odd $5$-periodic function such that $f(1)=1$, $f(2)=-2$, then $D_2(1,f)=0$. Indeed,
\est{
\sum_{n\in\Z \atop n\equiv 1\mod 5}\frac{d(|n|)}n=2\sum_{n\in\Z \atop n\equiv 2\mod 5}\frac{d(|n|)}n=\frac{4\pi^2}{25\sqrt 5}
}
(cf.~\eqref{id-xr} below), where the sums have to be interpreted as the limits as $X\to\infty$ of their truncations at $|n|\leq X $. 
Similarly, if $f$ is the odd $13$-periodic function such that 
\begin{align*}
&f(1)=18a, && f(4)=18b,&& f(3)=18c\\
& f(2)=19a + 11b + 4c,&& f(8)=-4 a + 19 b + 11 c,&& f(6)=-11 a - 4 b + 19 c
\end{align*}
for any $a,b,c\in\C$, then $D_2(1,f)=0$.
Notice the pattern and that the ordering we chose is not casual: indeed mod $13$ we have $(2^0,2^2,2^4)\equiv (1,4,3)$ and $(2^1,2^3,2^5)\equiv (2,8,6)$, with $2$ a primitive root mod $13$. 

The above examples are far from being unique. Indeed, if $(k,p-1)>1$, then one has no non-trivial solutions to $D_k(1,f)=0$, with $f:\Z\to\Q$ odd and periodic mod $p$, if and only if $(k,p-1)=2$ and $p\equiv 3\mod4$.
We classify the possible cases in the following Theorem, which generalizes the result of Chowla corresponding to the case $k=1$.

\begin{theo}\label{mt}
Let $k\in\N$, $p$ be an odd prime and let $K$ be a number field with $K\cap\Q(\xi_p)=\Q$. Let $V$ be the vector space over $K$ consisting of odd $p$-periodic functions $f:\Z\to K$ and let $V_0$ be the subspace $V_0:=\{f\in V\mid L_k(1,f)=0\}$. Then, 
\es{\label{me}
\dim_K(V_0)\geq 
\begin{cases}
\dim_K(V)\frac{r-1}r & \text{if $v_2(p-1)> v_2(k)$,}\\
\dim_K(V)\frac{r-2}r & \text{if $v_2(p-1)\leq v_2(k)$,}
\end{cases}
}
where $r=(k,p-1)$ and $v_2(a)$ denotes the $2$-adic valuation of $a$.
Moreover, the equality holds if $(k,p-1)\leq 2$ or if $(k,p-1)=4$ and $p\equiv 5\mod 8$. In particular,  $\dim_K(V_0)=0$ if and only if $(k,p-1)=1$ or if $(k,p-1)=2$ and $p\equiv 3\mod 4$.
\end{theo}

In the cases $(k,p-1)=2$ and $p\equiv 1\mod 4$ or $(k,p-1)=4$ and $p\equiv 5\mod 8$ we shall also show that $L(1,f)\neq0$ whenever $f\not\equiv 0$ has support entirely contained in the set of square residues mod $p$ (or, analogously, of square non-residues). As a consequence of this and of Theorem~\ref{mt}  we will deduce the following.

\begin{theo}\label{mt2}
Let $k\in\N$, $p$ be an odd prime with either $(k,p-1)\leq 2$ or $p\equiv 5\mod 8$ and $(k,p-1)=4$.
Then the set of values $L(1,\chi)^k$ are linearly independent over $\Q$ for $\chi$ that runs through the odd Dirichlet characters mod $p$. 
Moreover, under Schanuel's conjecture the same result holds true also when $\chi$ varies among all non-principal Dirichlet characters mod $p$.
\end{theo}

It seems likely that the equality in~\eqref{me} (as well as a suitable modification of Theorem~\ref{mt2}) holds true with no conditions on $(k,p-1)$; in order to prove this one would need to show that certain explicit linear combinations of $k$-th powers of Dirichlet $L$-functions are non-zero.

At first sight Theorem~\ref{mt} doesn't seem to say anything about the interesting case of the odd part of the Estermann function at $s=1$, $D_{\sin}(1,a/p):=\sum_{n\geq 1}\frac{d(n)\sin(2\pi na/p)}n$, where $(a,p)=1$. Indeed, the number field generated by $\sin(2\pi/p),\dots,\sin(2\pi (p-1)/2p))$ has a non-trivial intersection with $\Q(\xi_p)$. 
However in fact one has (see \cite{CW} Hilfsatz 14 or  \cite{lBT} Theorem 4.4)  
$D_{\sin}\pr{1,\frac ap}=-\pi\sum_{n\geq1}\frac{B(an/p)}{n},
$ with 
$B(x)= \{ x \}-\frac12$ for $x\notin\Z$ and $B(x)=0$ otherwise, 
where $\{x\}$ denotes the fractional part of $x$.  
Thus, the non-vanishing of  $D_{\sin}(1,\frac ap)\neq0$ follows directly from Chowla's result (i.e. Theorem~\ref{mt} with $k=1$).

The proof of Theorem~\ref{mt} is in fact a variation of Chowla's proof in~\cite{Cho2}. In this proof he showed that the values $\cot(\frac{\pi}p),\dots,\cot(\frac{\pi(p-1)}{2p})$ are linearly independent over $\Q$  by proving that if $g$ is a generator of $(\Z/p\Z)^*$ then the determinant of the matrix $(\cot(\frac \pi p g^{2(i+j)}))_{1\leq i,j\leq ({p-1})/2}$ is a non-zero multiple of the relative class number $h_p^-$. One then obtains the result on the non-vanishing of $L(1,f)$ from the fact that $L(1,f)$ can be written as a linear combination in $\cot(\frac{\pi}p),\dots,\cot(\frac{\pi(p-1)}{2p})$.

In our case, the analogue of the cotangent function is given by the sums
\est{
x_k(r;p):=\frac{1}{p^k}\sumstar_{m_1,\dots,m_k\mod p\atop m_1\cdots m_k\equiv r\mod p}\cot\pr{\pi\tfrac {m_1}p}\cdots \cot\pr{\pi \tfrac{m_k}{p}},
}
for $p,k\in\N$ and $r\in\Z$, and where $\sumstar$ indicates that the sum is restricted to coprime moduli mod $p$. Note that $r\mapsto x_k(r,p)$ is odd. 
We notice that $x_k(r;p)$  is reminiscent of several other arithmetic objects. For example in the case $k=2$ (and ignoring the difference in the normalizations) if we replace $m_1m_2\equiv r$ by  $m_1\equiv r m_2$, we obtain the Dedekind sum, whereas if we replace one of the cotangents $\cot(\frac{\pi x}p)$ by its discrete Fourier transform (i.e. essentially the fractional part $\{\frac xp\}$), then we obtain the Vasyunin sum (for which see e.g.~\cite{Vas,BC}). The closest analogy, however, is with the hyper-Kloosterman sum $K_k(r;p)$, which is obtained by replacing $\cot(\frac{\pi x}p)$ by $\e{\frac xp}$. Indeed, for $k$ even both $x_k(\cdot \, ; p)$ and $K_k(\cdot \, ; p)$ take values in the real cyclotomic field $\Q(\xi_p)^+$, where $\Q(\xi_n)^+:=\Q(\xi_n+\xi_n^{-1})$, and behave in the same way with respect to the action of the Galois group $\Gal(\Q(\xi_p)/\Q)$ (for $k$ odd $x_k(r,p)\in\Q(\xi_{4p})^+$). 
More precisely, and analogously to what happen for the Kloosterman sums, it is easy to see (c.f. Corollary~\ref{field}) that if $H$ is the subgroup of order $(k,p-1)$ of $\Gal(\Q(\xi_p)/\Q)\sim (\Z/p\Z)^*$, then $i^k x_k(r;p)$ is in $\Q(\xi_p)^H$, the subfield fixed by $H$. If one could show that $x_k(r;p)\neq \pm x_k(\ell;p)$ for all $r\not\equiv \pm \ell\mod p$, than one would obtain that each of the values $x_k(r;p)$ for $(r,p)=1$ generates the aforementioned fixed fields.  We refer to~\cite{Fis,Wan} for some results on the algebraic properties of $K_k(\cdot\, ; p)$ related to this and to Theorem~\ref{algt} below. 

\begin{theo}\label{algt}
Let $p$ be a prime, let $k\in\N$ and let $K$ be a number field such that $K\cap\Q(\xi_p)=\Q$.
Then the values $x_k(1;p),\dots, x_k((p-1)/2;p)$ are linearly independent over $K$ if and only if $(k,p-1)=1$ or if $(k,p-1)=2$ and $p\equiv 3\mod 4$. 
If $p\equiv 1\mod 4$ and $(k,p-1)=2$ or if $p\equiv 5\mod 8$ and $(k,p-1)=4$, then the values of each of the sets $S_{\pm}:=\{x_k(r;p)\mid r\leq \frac{p-1}2, (\frac rp)=\pm1\}$ are linearly independent over $K$, where  $(\frac rp)$ is the Legendre symbol. 
\end{theo}

We also mention that, as all the other aforementioned sums, $x_k(r;p)$ has some nice arithmetic features. For example, for $p\equiv 3\mod 4$, $k=2$ and $(r,p)=1$ one has 
\est{
\Tr_{\Q(\xi_p)^+/\Q}(x_k(r;p))=8 \pr{\frac rp}h(-p)^2 p^{- 1},
}
where $h(-p)$ is the class number of $\Q(\sqrt{-p})$ (c.f. Corollary~\ref{classnum}).

We conclude this introduction by giving an alternative ``analytic'' expression of $x_k(r;p)$, which is what will allow us to prove the above Theorems. Also, from this formula one can easily deduce the asymptotic for the moments of $x_k(r;p)$.
\begin{prop}\label{pass}
Let $k\in\N$, $p$ be a prime and $r\in\Z$ with $(r,p)=1$. Then
\es{\label{id-xr}
x_k(r;p)=\frac 12\pro{\frac{2}{\pi }}^k \sum_{n\in\Z\atop n\equiv r\mod p}\frac{d_k(|n|)}{n}.
}
In particular, if $f:\Z\to\C$ is odd and periodic modulo $p$, then
\es{\label{fff1}
D_k(1,f)=2\pro{\frac{\pi }2}^k\sum_{r=1}^{(p-1)/2}f(r)x_k(r;p).
}
\end{prop}
\begin{corol}\label{moments}
Let $p,k,m\in\N$ with $p$ be prime. Then for $m\geq2$ even we have
\est{
\sum_{r\mod p}x_k(r;p)^m=
\pro{\frac{2^{k-1}}{\pi^{k}}}^{m}\sum_{n\geq1}\frac{d_k(n)^{m}}{n^{m}}
+O_{m,k,\eps}(p^{-1+\eps}),
}
whereas for $m$ odd the left hand side is trivially $0$.

\end{corol}
\subsubsection*{Acknowledgement}
This works was started when both authors were visiting the Centre de Recherches Math\'ematiques in Montr\'eal. Both authors want to thank this institution for providing excellent working conditions.  Also, we thank St\'ephane Louboutin for useful suggestions. The work of the first author is partially supported by PRIN ``Number Theory and Arithmetic Geometry". The second author was a member of UMI 3457 supported by CNRS funding and is also supported by the ANR (grant ANR-14CE34-0009 MUDERA). 
\section{The sum $x_k(r;p)$}
\begin{lemma}\label{algeb}
Let $p\geq3$ be prime and let $k\in\N$. Then, $i^k x_k(r;p)\in\Q(\xi_p)$. More precisely, if $k$ is even then $ x_k(r;p)\in\Q(\xi_p)^+$, whereas if $k$ is odd then  $x_k(r;p)\in\Q(\xi_{4p})^+$. 
Moreover, for all $(c,p)=1$, let $\sigma_c$ be the automorphism of $\Q(\xi_{p})$ sending $\xi_{p}\mapsto \xi_p^{c}$. Then, $\sigma_c(i^k x_k(r;p))=i^k x_k(c^kr;p)$ for all $r\in\Z$.
\end{lemma}
\begin{proof}
By definition we have
\est{
x_k(r;p):=\frac{i^k}{p^k}\sum_{ m_1\cdots m_k\equiv r\mod p}\frac{\e{\tfrac {m_1}p}+1}{\e{\tfrac {m_1}p}-1}\cdots \frac{\e{\tfrac { m_k}p}+1}{\e{\tfrac { m_k }p}-1},
}
and so $x_k(r;p)\in i^k\Q(\xi_p)$. Now we have $\Q(\xi_p,i)=\Q(\xi_{4p})$ and so, since $x_k(r;p)\in\R$, we have $x_k(r;p)\in\Q(\xi_p)\cap\R=\Q(\xi_p)^+$ for $k$ even and $x_k(r;p)\in\Q(\xi_{4p})^+$ for $k$ odd.
Also,
\est{
\sigma_c(i^kx_k(r;p))&=\frac{(-1)^k}{p^k}\sum_{ m_1\cdots m_k\equiv r\mod p}\frac{\e{ \tfrac {cm_1}p}+1}{\e{\tfrac {cm_1}p}-1}\cdots \frac{\e{\tfrac {c m_k}p}+1}
{\e{\tfrac {c m_k }p}-1}=i^k x_k(c^kr;p),
}
by making the change of variables $m_i\to  \overline cm_i$ for each $i=1,\dots,k$. 
\end{proof}
\begin{corol}\label{field}
Let $k,p\in\N$ with $p$ prime and let $r\in\Z$. Then, $i^k x_k(r;p)\in\Q(\xi_p)^H$, where $H$ is the subgroup of order $(k,p-1)$ of $\Gal(\Q(\xi_p)/\Q)$.
\end{corol}
\begin{proof}
It's well known that $\Gal(\Q(\xi_p)/\Q)\sim(\Z/p\Z)^*$ is cyclic so $H$ is well defined, the Corollary then follows immediately from Lemma~\ref{algeb}.
\end{proof}

We now give a proof of Proposition~\ref{pass} and then compute the trace of $x_k(r;p)$.
\begin{remark}
For the sake of simplicity we shall ignore convergence issues when manipulating the order of summation of conditionally convergent series. One could make every step rigorous in several ways, for example by some analytic continuation arguments, or by using the ``approximate'' functional equations for the various series (in the form of exact formulae).
\end{remark}
\begin{proof}[Proof of Proposition~\ref{pass}]
We have
\es{\label{mfor}
\sum_{n\in\Z\atop n\equiv r\mod p}\frac{d_k(|n|)}{n}&=\frac1{\varphi(p)}\sum_{\chi\mod p}\overline\chi(r)\sum_{n\in\Z_{\neq0}}\frac{d_k(|n|)\chi(n)}{n}\\
&=\frac1{\varphi (p)}\sum_{\chi\mod p}\overline\chi(r)(1-\chi(-1))L(1,\chi)^k.
}
Now, if $\chi$ is a primitive odd character, then we have
\begin{align*}
L(1,\chi)&=\frac{\pi i\tau(\chi)}p\sumstar_{a\mod p}\overline\chi(a)\pg{\frac ap} 
=  \frac{\pi i\tau(\chi)}p\sumstar_{a\mod p}\overline\chi(a)
{B\pr{\frac ap}}\\ 
&=\frac{\pi i}p\sumstar_{a,m\mod p}\overline\chi(a)
{B\pr{\frac ap}}\chi(m)\e{\frac {m}p},
\end{align*}
where $\tau(\chi)=\sum_{m \mod p} \chi(m)\e{m/p}$ is the Gauss sum (see~\cite{Was}, p. 36). Now, for $(m,p)=1$ we have
\est{
\sumstar_{a\mod p}
{B\pr{\tfrac{a}{p}}}\e{\tfrac{ma}p}=-\frac i2\cot\pr{\pi \tfrac{ m}p}
}
and so, after making the change of variables $m\to ma$, we have
\es{\label{effl}
L(1,\chi)=\frac{\pi }{2p}\sumstar_{m\mod p}\chi(m)\cot\pr{\pi \tfrac{ m}p}.
}
Noticing that the right hand side is zero if $\chi$ is even, we then have
\est{
\sum_{n\in\Z\atop n\equiv r\mod p}\frac{d_k(|n|)}{n}=\frac 2{\varphi(p)}\prbigg{\frac{\pi }{2p}}^k\sum_{\chi\mod p}\overline\chi(r)\prBigg{\sumstar_{m\mod p}\chi(m)\cot\pr{\pi \tfrac{ m}p}}^k.
}
The result then follows by expanding the power and exploiting the orthogonal relation for Dirichlet characters.
\end{proof}
\begin{corol}\label{classnum}
Let $k\in\N$ and $p\geq3$ be prime; let $r\in\Z$ with $(r,p)=1$. Then if $v_2(p-1)> v_2(k)$, then $\Tr_{\Q( \xi_{p})/\Q}(i^kx_k(r;p))=0$. Otherwise, we have
$$\Tr_{\Q(\xi_{p})^+/\Q}(x_k(r;p))=\frac 14\pro{\frac{2}{\pi }}^k\sumstar_{\chi\mod p\atop\chi^k=1}(1-\chi(-1))\overline\chi(r)L(1,\chi)^k
.$$ 
In particular, if $(k,p-1)=2$ and $p\equiv 3\mod 4$, then 
\est{
\Tr_{\Q(\xi_{p})^+/\Q}(x_k(r;p))=2^{2k-1}\pro{\frac rp}h(-p)^k p^{-  k/2}.
}
\end{corol}
\begin{proof}
By Lemma~\ref{algeb} and Proposition~\ref{pass}, given a generator $g$ of $(\Z/p\Z)^*$ we have
\est{
\Tr_{\Q(\xi_{p})/\Q}(i^k x_k(r;p))=\sum_{j=1}^{p-1}\sigma_{g^j}(i^k x_k(r;p))=\frac {i^k}2\pro{\frac{2}{\pi }}^k \sum_{j\mod {p-1}}\sum_{n\in\Z\atop n\equiv r g^{kj}\mod p}\frac{d_k(|n|)}{n}.
}
Now, if $v_2(p-1)> v_2(k)$, writing $k=2^{v_{2}(k)}k'$ and making the change of variables $j\to\frac{p-1}{2^{v_{2}(k)+1}}+j$ we have that the condition on the sum over $n$ becomes $n\equiv rg^{kj+k'\frac{p-1}2}\equiv -rg^{kj}\mod p$. Thus, making the change $n\to -n$ we obtain the opposite of the original sum and so $\Tr_{\Q(\xi_{p})/\Q}(i^kx_k(r;p))=0$. Now assume $v_2(p-1)\leq v_2(k)$, notice that in particular $k$ is even.
By~\eqref{mfor} we have
\est{
\sum_{j\mod {p-1}}\sum_{n\in\Z\atop n\equiv rg^{kj}\mod p}\frac{d_k(|n|)}{n}=\frac1{\varphi (p)}\sum_{j\mod {p-1}}\sumstar_{\chi\mod p}\overline\chi(r)\overline\chi(g)^{kj}(1-\chi(-1))L(1,\chi)^k.
}
Taking the sum over $j$ inside, we obtain $0$ unless $\overline\chi(g)^{k}=1$, i.e. if $\chi$ is a character of order dividing $(k,p-1)$. Thus,
\est{
\Tr_{\Q(\xi_{p})^+/\Q}(x_k(r;p))=\frac12\Tr_{\Q(\xi_{p})/\Q}(x_k(r;p))=\frac 14\pro{\frac{2}{\pi }}^k\sumstar_{\chi\mod p\atop\chi^k=1}(1-\chi(-1))\overline\chi(r)L(1,\chi)^k.
}
If $(k,p-1)=2$ and $p\equiv 3\mod 4$, then the only character contributing to the sum is the quadratic character $\pr{\frac{\cdot }{p}}$ and so 
\est{\Tr_{\Q(\xi_{p})^+/\Q}(x_k(r;p))=\frac 12\pro{\frac{2}{\pi }}^k\pro{\frac rp}L(1,\pr{\frac{\cdot }{p}})^k=2^{2k-1}\pro{\frac rp}h(-p)^k p^{- k/2},
} 
by the class number formula.
\end{proof}

We now give a proof of Corollary~\ref{moments}. We don't give many details as the proof is very similar (and actually a bit simpler than) to the proof of Theorem~1 of~\cite{Bet}.
\begin{proof}[Proof of Corollary~\ref{moments}]
Let $\chi\mod p$ be an odd character. Using the functional equation 
\est{
\Lambda(s,\chi) := \left(p/\pi\right)^{(s+1)/2}
\Gamma(\tfrac {1+s}2) L(s,\chi)={\frac{\tau(\chi)}{ip^{1/2}}}\Lambda(1-s,\overline \chi), 
}
and proceeding as in~\cite{Bet} we obtain that for every $\eps>0$  
\es{\label{applc}
L(1,\chi)^k
=\sum_{n=1}^\infty\frac{d_k(n)\chi(n)}ng(n/p^{2k})+O_{\eps,k}(p^{-5k/2}).\\
}
where $g(x)$ is a smooth function such that $g(x)=1$ for $0\leq x\leq 1$, $0\leq g(x)\leq 1$ for $1\leq x\leq 2$, and $g(x)=0$ for $x>2$. 
Now, applying~\eqref{id-xr}, \eqref{mfor} and~\eqref{applc} and then going back to the congruence condition we have
\est{
2\pro{\frac{\pi }{2}}^k x_k(r;p)=\sum_{n\in\Z\setminus\{0\}\atop n\equiv r\mod p}\frac{d_k(|n|)}ng(|n|/p^{2k})+O_{\eps,k}(p^{-5k/2}).
}
Thus, using the bound $d_k(n) \ll_{\eps,k} n^\eps$, we have 
\est{
\sum_{r\mod p}\pro{\frac{\pi^k x_k(r;p)}{2^{k-1}}}^{2m}=\sum_{\substack{(n_1,\dots,n_{2m})\in(\Z\setminus\{0\})^{2m}\\n_i\equiv n_j\mod p\\\forall 1\leq i< j\leq 2m}}\frac{\tilde d_k(n_1,\dots,n_{2m})}{n_1\cdots n_{2m}}+O_{k,m,\eps}(p^{-\frac{5k}2+1+\eps})
}
where $\tilde d_k(n_1,\dots, n_{2m}):=d_k(|n_1|)\cdots d_k(|n_{2m}|)g(|n_1|/p^{2k})\cdots g(|n_{2m}|/p^{2k})$. The contribution of the terms with $n_a\neq n_b$ for some $a\neq  b$ is trivially
\est{
\ll_{k,m,\eps}\sum_{\substack{0<|n_a|,|n_b|<2p^{2k}\\n_a\equiv n_b\mod p\\ n_a\neq n_b}}\frac{ p^{4mk\eps}}{n_an_b}\ll_{k,m,\eps} p^{(4mk+2)\eps-1}.
}
It follows that
\est{\sum_{r\mod p}\pro{\frac{\pi^k x_k(r;p)}{2^{k-1}}}^{2m}=\sum_{n\in \Z\setminus\{0\}}\frac{\tilde d_k(n,\dots,n)}{n^{2m}}+O_{k,m,\eps}(p^{-1+\eps}).
}
Since $\sum_{n\geq p^{2k}}(d_k(n)/n)^{2m}\ll_{k,\eps} p^{(1-(1-\eps)2m)2k}$, we can remove the contribution of the functions $g$ in $\tilde d_k$ at a negligible cost and we obtain the claimed result.
\end{proof}

\section{The vanishing of $D_k(1,f)$}
\begin{lemma}\label{conv}
Let $p\geq3$ be prime and let $k\in\N_{\geq 2}$. Let $f:\Z\to\C$ be periodic mod $p$. Then, $D_k(s,f)$ is entire if and only if $f(0)=0$ and $\Av(f)=0.$ Also, in this case we have
\es{\label{fff}
D_k(s,f)&=\frac1{\varphi(p)}\sumstar_{\chi\mod p}c_\chi(f) L(s,\chi)^{k},\\
}
where $c_\chi(f):=\sum_{r\mod p}f(r) \overline\chi(r)$.
\end{lemma}
\begin{proof}
For $\Re(s)>1$ we have
\begin{align}
D_k(s,f)&=\sum_{r\mod p}f(r)\sum_{n\in\N\atop n\equiv r\mod p}\frac{d_k(n)}{n^s}\notag\\
&=f(0)\sum_{n\in\N}\frac{d_k(pn)}{(pn)^s}+\sumstar_{r\mod p}\frac{f(r)}{\varphi(p)}\sum_{\chi\mod p}\overline\chi(r)L(s,\chi)^{k}\notag\\
&=f(0)\sum_{n\in\N\atop p|n}\frac{d_k(n)}{n^s}+L(s,\chi_0)^k\sumstar_{r\mod p}\frac{f(r)}{\varphi(p)}
+\frac1{\varphi(p)}\sumstar_{\chi\mod p}c_\chi(f) L(s,\chi)^{k},\label{ffdk}
\end{align}
where $\chi_0$ is the principal character mod $p$. Now, 
\est{
\sum_{n\in\N\atop p|n}\frac{d_k(n)}{n^s}=\pr{1-(1-p^{-s})^k}\zeta(s)^k,\qquad L(s,\chi_0)=(1-p^{-s})\zeta(s)
}
and so
\est{
D_k(s,f)&=\prbigg{f(0)+ (1-p^{-s})^k\prbigg{\sumstar_{r\mod p}\frac{f(r)}{\varphi(p)}-f(0)}}\zeta(s)^k\\
&\quad 
+\frac1{\varphi(p)}\sumstar_{\chi\mod p}c_\chi(f) L(s,\chi)^{k}.
}
In particular, $D_k(s,f)$ is meromorphic on $\C$ with possibly a pole in $s=1$ only. Moreover, $D_k(s,f)$ is entire if and only if 
\est{
P(x):=f(0)+(1-x)^k\prbigg{\sumstar_{r\mod p}\frac{f(r)}{\varphi(p)}-f(0)}
}
has a zero of order $k$ at $x=1/p$. Thus, since $k\geq2$ we have that $P(x)$ has discriminant equal to zero. Now, the discriminant of $P$ is
$$\Delta=k^k f(0)^{k-1}\prbigg{\sumstar_{r\mod p}\frac{f(r)}{\varphi(p)}-f(0)}^{k-1}
$$
and so we must have that either $f(0)$ or the term in the big brackets is zero. Then, imposing $P$ has a zero at $1/p$ we find that both $f(0)$ and the term in the brackets need to be zero, as desired. Equation~\eqref{fff} then follows immediately from~\eqref{ffdk}.
\end{proof}

We now prove Proposition~\ref{sch}.
\begin{proof}[Proof of Proposition~\ref{sch}]
For $k=1$, the result was proven in~\cite{BBW} so assume $k\geq2$.
Let $f_{odd}(n):=(f(n)-f(-n))/2$ and $f_{even}(n):=(f(n)+f(-n))/2$ so that $D_k(1,f)=D_k(1,f_{odd})+D_k(1,f_{even})$. Then, one easily checks that $c_{\chi}(f_{even})=0$ for $\chi$ odd and $c_{\chi}(f_{odd})=0$ for $\chi$ even. Thus, since $L(1,\chi)\in\pi\overline\Q$ for $\chi$ odd (see formula \eqref{effl}), then by~\eqref{fff} also $D_k(1,f_{odd})\in\pi^k\overline\Q$. Now, by Schanuel's conjecture we have that $\pi$ and the values of $L(1,\chi)$ for $\chi\mod p$ even are algebraically independent over $\Q$ (this is stated in the paragraph after Corollary~2 in~\cite{MM}, and essentially proved in Section~4 therein, without including $\pi$; however the same proof allows one to include $\pi$ since $\log(-1)=\pi i$ when choosing the branch for the logarithm suitably). Thus we could have $D_k(1,f_{odd})=-D_k(1,f_{even})$ only if $c_{\chi}(f_{even})=0$ for all $\chi$ even, i.e. if $f_{even}=0$ and so if $f$ is odd.
\end{proof}

By Proposition~\ref{sch}, at least conditionally, in order to find functions $f:\Z\to\overline\Q$ such that $D_k(1,f)=0$ we need to take $f$ odd. Then, for $f$ odd with $D_k(1,f)=0$
by~\eqref{fff1} we have
\es{\label{bfo}
\sum_{r=1}^{(p-1)/2}f(r)x_k(r;p)=\frac12\sum_{j=0}^{p-2}f(g^j)x_k(g^j;p)=0,
}
where $g$ is any generator of $(\Z/p\Z)^*$ and where we used that  $g^{j+\frac{p-1}2}=-g^j$. 
If $f(1),\dots f(p-1)\in K$ with $K\cap\Q(\xi_{p})=\Q$ (if $k$ is even $K\cap\Q(\xi_{p})^+=\Q$ would suffice), then we can extend the automorphism $\sigma_c$ defined in Lemma~\ref{algeb} to an automorphism of $K(\xi_{p})$ such that $\sigma_c$ acts trivially on $K$ ({see \cite{Rom} Corollary~6.5.2 p.~161}). 
By a slight abuse of notation we still indicate the automorphism by $\sigma_c$. Then, multiplying~\eqref{bfo} by $i^k$ and applying $\sigma_{g^{\ell}}$ we obtain new conditions for $f$:
\es{\label{bf}
\sum_{j=0}^{p-2}f(g^j)x_k(g^{k\ell+j};p)=0
}
for all $\ell\in\Z$. It is clearly sufficient to take $0\leq \ell<\frac{p-1}u$ where $u=(k,p-1)$ and in fact, if $v_2(p-1)>v_2(k)$ then we can take $0\leq \ell<\frac{p-1}{2u}$ since the following $\frac{p-1}{2u}$ equations are just the negative of the first ones.  
Thus we have a system of linear equations in the values of $f$. We now study the determinants of the relevant matrices for such system.
\begin{lemma}\label{linalg}
Let $m\geq1$.
For $\boldsymbol v=(v_0,\dots, v_{m-1})$, let 
\est{
A_+(\boldsymbol v):=\!\begin{pmatrix}
v_0 & v_1 & \dots & v_{m-2} & v_{m-1}\\
v_1 & v_2 & \dots & v_{m-1} & v_{0}\\
\vdots & \vdots && \vdots &\vdots \\
v_{m-1} & v_0 & \dots & v_{m-3} & v_{m-2}\\
\end{pmatrix}\quad
A_-(\boldsymbol v):=\!\begin{pmatrix}
v_0 & v_1 & \dots & v_{m-2} & v_{m-1}\\
v_1 & v_2 & \dots & v_{m-1} & -v_{0}\\
\vdots & \vdots && \vdots &\vdots \\
v_{m-1} & -v_0 & \dots & -v_{m-3} & -v_{m-2}\\
\end{pmatrix}.
}
Then, 
\est{
\det (A_+(\boldsymbol v))&=(\sin(\tfrac {\pi m}2)- \cos(\tfrac {\pi m}2)) \prod_{\ell=0}^{m-1} \pr{\sum_{j=0}^{m-1}v_j\xi_{m}^{j\ell}},\\ 
\det (A_-(\boldsymbol v))&=(\sin(\tfrac {\pi m}2)+ \cos(\tfrac {\pi m}2)) \prod_{\ell=0\atop \ell\text{ odd}}^{2m}\pr{\sum_{j=0}^{m-1}v_j\xi_{2m}^{j\ell}}.
}
Also, for $j\in\{0,\dots, m-1\}$ and $v = (v_0,v_1,\ldots,v_{m-1})\in\C^m$, let 
\est{\boldsymbol v_j^{\pm}=v_j^{\pm}(v):=(v_j,\dots, v_{m-1},\pm v_0,\dots, \pm v_{j-1})
,}
and let $\boldsymbol v_j^{-}=-\boldsymbol  v_{j-m}^{-}$ for $j\in\{m,\dots, 2m-1\}$. Also, let
$\boldsymbol u_j^{\pm} = \boldsymbol v_j^{\pm}(\boldsymbol u)$ with $\boldsymbol u=(1,0,\dots,0)\in\C^m$. 
 Then $A_\pm(\boldsymbol v_j^{\pm})=C_\pm(\boldsymbol u_j^{\pm})A_\pm(\boldsymbol v)$, where $C_\pm$ is a matrix defined in the proof.
\end{lemma}
\begin{proof}
One can easily check that $A_\pm(\boldsymbol v)A_\pm(\boldsymbol u)=C_{\pm}(\boldsymbol v)$, where
\est{
C_+(\boldsymbol v):=\begin{pmatrix}
v_0 & v_{m-1} & \dots & v_{2} & v_1\\
v_1 & v_0 & \dots & v_{3} & v_{2}\\
\vdots & \vdots && \vdots &\vdots \\
v_{m-1} & v_{m-2} & \dots & v_{1} & v_{0}\\
\end{pmatrix}\qquad 
C_-(\boldsymbol v):=\begin{pmatrix}
v_0 & -v_{m-1} & \dots & -v_{2} & -v_1\\
v_1 & v_0 & \dots & -v_{3} & -v_{2}\\
\vdots & \vdots && \vdots &\vdots \\
v_{m-1} & v_{m-2} & \dots & v_{m-1} & v_{0}\\
\end{pmatrix}.
}
Similarly, one shows that the identity $A_\pm(\boldsymbol v_j^{\pm})=C_\pm(\boldsymbol u_j^{\pm})A_\pm(\boldsymbol v)$ holds.
The eigenvectors of $C_+(\boldsymbol v)$ are 
 $\boldsymbol c_\ell^+=(1,\xi_m^{(m-1)\ell},\xi_m^{(m-2)\ell},\dots,\xi_m^{\ell})^T$ for $0\leq \ell<m$ where $T$ indicates the transpose, whereas the eigenvectors of $C_-(\boldsymbol v)$ are $\boldsymbol c_r^-=(\xi_{2m}^{mr},\xi_{2m}^{(m-1)r},\dots,\xi_{2m}^{r})^T$ with  $1\leq r <2m $ and $r$ odd. The eigenvalues are given by
\est{
C_+(\boldsymbol v) \boldsymbol c_\ell^+=\pr{\sum_{j=0}^{m-1}v_j\xi_{m}^{j\ell}}\boldsymbol c_\ell^+,\qquad
C_-(\boldsymbol v) \boldsymbol c_r^-=\pr{\sum_{j=0}^{m-1}v_j\xi_{2m}^{jr}}\boldsymbol c_r^-.
} 
The statement on the determinants then follows since $\det (A_\pm(\boldsymbol u))=\sin(\frac {\pi m}2)\mp \cos(\frac {\pi m}2)$.
\end{proof}

\begin{lemma}\label{det1}
Let $p\geq3$ be prime. Let $k\in\N$ and assume $v_2(p-1)> v_2(k)$. Write $k=k'u$ with $(k,p-1)=u$ and let $p-1=uv$ (so that $v$ is even). For $(r,p)=1$, let $M_{g,r}$ be the matrix $M_{g,r}:=(\sigma_g^{i+j}(x_{k}(r;p)))_{0\leq i,j<v/2}$ where $g$ is a generator of $(\Z/p\Z)^*$. Then
\est{
\det (M_{g,r})=(\sin(\tfrac {\pi v}{4})+ \cos(\tfrac {\pi v}{4}))\overline \chi_*(r)^{v^2/4}\frac{2^{(k-1)v/2}}{\pi^{kv/2} }\prod_{\ell=0\atop \ell \text{ odd}}^{v}\pro{\frac{1}{u}\sum_{a\mod u}\overline \chi_*(r)^{va}L(1,\chi_*^{\ell+va})^k},
}
for a generator $\chi_*$  of the group of characters mod $p$.
Moreover, for all $j\in\Z$ we have $\sigma_{g^j}(M_{g,r})=C_-(\boldsymbol u_j^-)M_{g,r}$ with $C_-$ and $\boldsymbol u_j^-$ as in Lemma~\ref{linalg}.
\end{lemma}
\begin{proof}
Writing $t=g^{k'}$ we have that $t$ is also a primitive root mod $p$.
Thus, by
Lemma~\ref{algeb}, we have $M_{g,r}=(x_k(rt^{u(i+j)};p))_{0\leq i,j<v/2}$ and since 
\est{x_k(rt^{u(\ell+\frac v{2})};p)=x_k(rt^{u\ell+\frac{p-1}2};p)=x_k(-rt^{u\ell};p)=-x_k(rt^{u\ell};p)}
we have $M_{g,r}=A_-(\boldsymbol x)$, with 
$$\boldsymbol x=(x_k(r;p),x_k(rt^u;p),\dots, x_k(rt^{p-1-u};p)).$$
Thus, by Lemma~\ref{linalg} we have $\sigma_{g^j}(M_{g,r})=A_-(\boldsymbol x_j)=C_-(\boldsymbol u_j^-)A_-(\boldsymbol x)=C_-(\boldsymbol u_j^-)M_{g,r}$. Also,
\est{
\det (M_{g,r})=(\sin(\tfrac {\pi v}{4})+ \cos(\tfrac {\pi v}{4})) \prod_{\ell=0\atop \ell \text{ odd}}^{v} \prbigg{\sum_{j=0}^{v/2-1}x_k(rt^{ju};p)\xi_{v}^{j\ell}}.\\ 
}
Now,
\est{
\sum_{j=0}^{v/2-1}x_k(rt^{ju};p)\xi_{v}^{j\ell}&=\sum_{j=0}^{v/2-1}x_k(rt^{ju};p)\e{\frac{\ell\nu_t(t^{ju})}{p-1}}=\frac12\sum_{j=0}^{v-1}x_k(rt^{ju};p)\e{\frac{\ell\nu_t(t^{ju})}{p-1}},\\
}
where $\nu_t(c)$ is the minimum non-negative integer such that $t^{\nu_t(c)}\equiv c\mod p$. Then, writing $\eta(c):=\e{\frac{\nu_t(c) }{p-1}}$ if $(c,p)=1$ and $\eta(c):=0$ otherwise, we have that $\eta$ is a primitive odd character modulo $p$. Also, $\eta$ generates the group of characters mod $p$. Then, we re-write the above as
\est{
\sum_{j=0}^{v/2-1}x_k(rt^{ju};p)\xi_{v}^{j\ell}&=\frac12\sumstar_{c\mod p\atop u|\nu(c)}x_k(rc;p)\eta(c)^\ell=\frac1{2u}\sum_{a\mod u}\sumstar_{c\mod p}x_k(rc;p)\eta(c)^\ell\e{\frac{a\nu(c)}{u}}\\
&=\frac1{2u}\sum_{a\mod u}\sumstar_{c\mod p}x_k(rc;p)\eta(c)^{\ell+va}\\
&=\frac{\overline \eta(r)^{\ell}}{2u}\sum_{a\mod u}\overline \eta(r)^{va}\sumstar_{c\mod p}x_k(c;p)\eta(c)^{\ell+va}.\\
}
By Proposition~\ref{pass} with $f = \eta^{\ell+va}$ for $\ell$ odd this is
\est{
=\pro{\frac{2}{\pi }}^k \frac{\overline \eta(r)^{\ell}}{2u}\sum_{a\mod u}\overline \eta(r)^{va}L(1,\eta^{\ell+va})^k.
}
Finally, we have
\est{
\prod_{\ell=0\atop \ell \text{ odd}}^{v}\eta(r)^{\ell}=\eta(r)^{\sum_{\ell=0}^v(1-(-1)^\ell)\ell/2}=\eta(r)^{v^2/4}
}
and the result follows.
\end{proof}
\begin{corol}\label{fcd1}
With the same notation and conditions of Lemma~\ref{det1}, if $(k,p-1)=1$ we have
\est{
\det (M_{g,r})=(\sin(\tfrac {\pi (p-1)}{4})+ \cos(\tfrac {\pi (p-1)}{4}))\pro{\tfrac rp}^{(p-1)/2}\frac{2^{k(p-2)-\frac{p-1}2}(h_p^-)^k}{p^{\frac{k(p+3)}4}},
}
where $h_p^-$ is the relative class number of the 
field $\Q(\xi_p)$, that is  $h_p^-:= h_p/h_p^+$ where $h_p$ and $h_p^+$ are the class numbers of $\Q(\xi_p)$ and 
$\Q(\xi_p)^+$ respectively.
\end{corol}
\begin{proof}
First we observe that for a generator $\chi_*$ of the group of characters we have $\chi^{(p-1)/2}=(\tfrac \cdot p)$.
Then, by the proposition we have
\est{
\det(M_{g,r})&=(\sin(\tfrac {\pi (p-1)}{4})+ \cos(\tfrac {\pi (p-1)}{4})) \pro{\tfrac rp}^{(p-1)/2} \frac{2^{\frac {(p-1)(k-1)}2}}{\pi ^{\frac k2(p-1)}}\prod_{\ell=1\atop \ell \text{ odd}}^{p-1}L(1,\chi_*^\ell)^k.\\
}
Now, referring to~\cite{Was} (Chapters 3 and 4)  for the basic results on cyclotomic fields, we have for $s>1$,
\est{
\prod_{\ell=1\atop \ell \text{ odd}}^{p-1}L(s,\chi_*^\ell)=
\prod_{\chi \mod p \text{ odd}}L(s,{\chi})
= \frac{\prod_{\chi \mod p}L(s,{\chi})}
{\prod_{\chi \mod p \text{ even}}L(s,{\chi})}
=
\frac{\zeta_{\Q(\xi_p)}(s)}{\zeta_{\Q(\xi_p)^+}(s)},
}
where $\zeta_K(s)$ denotes the Dedekind the zeta-function corresponding to the field $K$. 
Then by the class number formula we have (c.f.~\cite{Was} p. 41-42)
\est{
\prod_{\ell=1\atop \ell \text{ odd}}^{p-1}L(1,\chi_*^\ell)=\pro{\frac{\pi^{p-1}(h_p/h_p^+)^22^{p-1}}{4p^2p^{p-2}/p^{(p-3)/2}}}^{\frac{1}2}.
}
The Corollary then follows. 
\end{proof}
\begin{corol}\label{fcd2}
With the same notation and conditions of Lemma~\ref{det1}, if $(k,p-1)=2$ and $p\equiv 1\mod 4$ we have
\est{
\det (M_{g,r})=(\sin(\tfrac {\pi (p-1)}{8})+ \cos(\tfrac {\pi (p-1)}{8}))\frac{2^{(k-2)(p-1)/4}}{\pi^{k(p-1)/4} }\prod_{\ell=0\atop \ell \text{ odd}}^{(p-1)/2}\pro{L(1,\chi_*^{\ell})^k+(\tfrac r p)L(1,(\tfrac \cdot p)\chi_*^{\ell})^k}.
}
\end{corol}

\begin{lemma}\label{det2}
Let $p\geq3$ be prime. Let $k\in\N$ and assume $v_2(p-1)\leq  v_2(k)$. Write $k=k'u$ with $(k,p-1)=u$ and let $p-1=uv$. For $(r,p)=1$, let $M_{g,r}'$ be the matrix $M_{g,r}':=(\sigma_g^{i+j}(x_{k}(r;p)))_{0\leq i,j<v}$ where $g$ is a generator of $(\Z/p\Z)^*$. Then
\est{
\det (M_{g,r}')=(-1)^{(v-1)/2}  \overline \chi_*(r)^{v^2}\pro{\frac{2}{\pi }}^{kv}\prod_{\ell=0\atop \ell \text{ odd}}^{2v}\prbigg{\frac{1}{u}\sum_{a\mod {u/2}}\overline \chi_*(r)^{2va}L(1,\chi_*^{\ell+2va})^k},
}
for a generator $\chi_*$ of the group of characters mod $p$.
Moreover, for all $j\in\Z$ we have $\sigma_{g^j}(M_{g,r}')=C_+(\boldsymbol u_j^+)M_{g,r}'$ with $C_+$ and $\boldsymbol u_j^+$ as in Lemma~\ref{linalg}. 
\end{lemma}
\begin{proof}
We proceed as in the proof of Lemma~\ref{det1} setting $t=g^{k'}$ and $M_{g,r}'=(x_k(rt^{u(i+j)};p))_{0\leq i,j<v}.$ Then, since $x_k(rt^{u(\ell+v)};p)=x_k(rt^{u\ell};p)$ we have $M_{g,r}'=A_+(\boldsymbol x)$, with 
$$\boldsymbol x=(x_k(r;p),x_k(rt^u;p),\dots, x_k(rt^{p-1-u};p)).$$
Thus, by Lemma~\ref{linalg} we have $\sigma_{g^j}(M_{g,r}')=C_+(\boldsymbol u_j^+)M_{g,r}'$ and
\est{
\det (M_{g,r})=(\sin(\tfrac {\pi v}{2})- \cos(\tfrac {\pi v}{2})) \prod_{\ell=0}^{v-1} \prbigg{\sum_{j=0}^{v-1}x_k(rt^{ju};p)\xi_{v}^{j\ell}}.\\ 
}
Now, since $v$ is odd then $\sin(\tfrac {\pi v}{2})- \cos(\tfrac {\pi v}{2})=(-1)^{(v-1)/2}$. Also, as in the proof of Lemma~\ref{det1} we have
\est{
\sum_{j=0}^{v-1}x_k(rt^{ju};p)\xi_{v}^{j\ell}=
{
\frac{\overline \eta(r)^{\ell}}{u}\sum_{a\mod u}\overline \eta(r)^{va}
\sumstar_{ c\mod p} x_k(c;p)\eta(c)^{\ell+va}.}
}
By symmetry the innermost sum is zero if $\ell+va$ is even and otherwise it is {$\pr{\tfrac2\pi}^kL(1,\eta^{\ell+va})^k$ 
by Proposition~\ref{pass}}.  Thus, 
\est{
\sum_{j=0}^{v-1}x_k(rt^{ju};p)\xi_{v}^{j\ell}=\pro{\frac{2}{\pi }}^k\frac{\overline \eta(r)^{\ell+\delta v}}{u}\sum_{a\mod {u/2}}\overline \eta(r)^{2va}L(1,\eta^{\ell+\delta v+2va})^k,
}
where $\delta=1$ if $2|\ell$ and $\delta=0$ otherwise, and the lemma follows.
\end{proof}
\begin{corol}\label{fcd1b}
With the same notation and conditions of Proposition~\ref{det2}, if $(k,p-1)=2$ and $p\equiv 3\mod 4$ we have
\est{
\det (M_{g,r}')= (-1)^{(p-3)/4}\pr{\tfrac rp} \frac{2^{k(p-2)-\frac{p-1}2}(h_p^-)^k}{p^{\frac{k(p+3)}4}}.
}
\end{corol}
\begin{proof}
One proceeds as for Corollary~\ref{fcd1}.
\end{proof}
\begin{corol}\label{fcd2b}
With the same notation and conditions of Proposition~\ref{det2}, if $(k,p-1)=4$ and $p\equiv 5\mod 8$ we have
\est{
\det (M_{g,r}')=(-1)^{(p-5)/8}  \overline \chi_*(r)^{(\frac {p-1}4)^2}\frac{2^{(k-1)(p-1)/4}}{\pi^{k(p-1)/4} }\prod_{\ell=0\atop \ell \text{ odd}}^{(p-1)/2}\pro{L(1,\chi_*^{\ell})^k+(\tfrac rp)L(1,(\tfrac \cdot p)\chi_*^{\ell})^k}.
}
\end{corol}

\begin{proof}[Proof of Theorem~\ref{mt} and of Theorem~\ref{mt2}]
First we show that the equality in~\eqref{me} and Theorem~\ref{mt2} hold if $(k,p-1)\leq 2$ or if $(k,p-1)=4$ and $p\equiv 5\mod 8$.

Let us begin with the case $(k,p-1)=1$ and assume $D_k(1,f)=0$ with $f$ odd.  As explained above, if $K\cap\Q(\xi_{p})=\Q$, then we have the system of equations~\eqref{bf}. Also, if $(k,p-1)=1$ then $t:=g^{k}$ is also a primitive root and so after a change of variable we can rewrite~\eqref{bf} as 
\est{
\frac12\sum_{j=0}^{p-2}f(t^{j})x_k(t^{\ell+j};p)=\sum_{j=0}^{(p-3)/2}f(t^{j})x_k(t^{\ell+j};p)=0.
}
for $0\leq \ell \leq \frac{p-3}2$. Equivalently, $M_{g,1}\boldsymbol f=\boldsymbol0$ where $\boldsymbol f=(f(t^0),\dots, f(t^{\frac {p-3}2}))^T$. Thus, since by Corollary~\ref{fcd1} we have $\det(M_{g,1})\neq0$, then the only solution to this system is $\boldsymbol f=\boldsymbol 0$, i.e. $f$ is identically zero. Thus we have proven that the equality holds in~\eqref{me} in this case.
\medskip 

Now, we prove Theorem~\ref{mt2} for the case $(k,p-1)=1$. Actually in this case we prove more generally that given a number field $K$ such that $K(\xi_{p-1})\cap \Q(\xi_p)=\Q$ then the values of $L(1,\chi)^k$ when $\chi$ runs over odd Dirichlet characters mod $p$ are linearly independent over  $K$. Assume $\sum^*_{\chi}a_{\chi}L(1,\chi)^k=0$, with $a_\chi\in K$ and $a_\chi=0$ if $\chi$ even. Then, writing $f:=\sum^*_{\chi}a_{\chi}\chi$ we have $D_k(1,f)=0$. Notice that $f$ is odd and takes values in the field $K(\xi_{p-1})$. Thus, by Theorem~\ref{mt} in the case $(k,p-1)=1$ we have $f\equiv 0$ and so $a_{\chi}$ for all $\chi$, as desired. 

By a similar argument as above, and by Proposition 1, we  can see that the statement in Theorem~\ref{mt2}  under Schanuel's conjecture follows from the unconditional case.  

\medskip

Next we prove the equality in~\eqref{me} and Theorem~\ref{mt2} when $(k,p-1)=2$, $p\equiv 3\mod 4$. Let $k=2k'$ so that $t= g^{k'}$ is a generator of $(\Z/p\Z)^*$. Since $-1$ is a quadratic non-residue mod $p$, $(\pm t^{2j})_{0\le j < \frac{p-1}2}$  spans all residues of $(\Z/p\Z)^*$. Thus we can rewrite the system~\eqref{bf} as 
\[
\sum_{j=0}^{\frac{p-3}{2}} f(t^{2j}) x_k( t^{2(j+\ell)};p) =0 
\] 
for $0\le \ell \le \frac{p-3}{2}$. 
Then we conclude as above, the only difference is that in this case the system is $M_{g,1}'\boldsymbol f'=\boldsymbol0$ 
with $\boldsymbol f'=(f(t^0), f(t^2),\dots, f(t^{p-3}))^T$
so that we apply Corollary~\ref{fcd1b}.

\medskip
Now, assume $(k,p-1)=2$ and $p\equiv 1\mod 4$ and write $k=2k'$ and $t=g^{k'}$. Then, after a change of variables~\eqref{bf} gives
\est{
\sum_{j=0}^{(p-3)/2}f(t^j)x_k(t^{2\ell+j};p)=0
}
 $0\leq \ell <\frac{p-1}4$. Equivalently, $M_{g,1}\boldsymbol f_1+M_{g,t}\boldsymbol f_t=\boldsymbol 0$, where $\boldsymbol f_r =(f(rt^0),f(rt^2),\dots, f(rt^{(p-3)/2}))$ for $r\in\{1,t\}$. By Corollary~\ref{fcd2}, we have
 \est{
\det (M_{g,1})&=\kappa\prod_{\ell=0\atop \ell \text{ odd}}^{(p-1)/2}\pro{L(1,\chi_*^{\ell})^k+(\tfrac 1 p)L(1,(\tfrac \cdot p)\chi_*^{\ell})^k}\\
&=\kappa\prod_{\ell=0\atop \ell \text{ odd}}^{(p-1)/2}\pro{L(1,\chi_*^{\ell})^{k'}+iL(1,(\tfrac \cdot p)\chi_*^{\ell})^{k'}}\pro{L(1,\chi_*^{\ell})^{k'}-i L(1,(\tfrac \cdot p)\chi_*^{\ell})^{k'}}
}
for some $\kappa\neq0$, $\chi_*$ a generator of the group of characters mod $p$.
Now, for $\ell$ odd  and $p\equiv 1\mod 4$ we have that both $\chi_*^{\ell}$ and $(\tfrac \cdot p)\chi_*^{\ell}$ are odd characters. Also, we have $\Q(\xi_{p-1})\cap\Q(\xi_p)=\Q$  and $(k',p-1)=1$. Thus, by the case $(k,p-1)=1$ of Theorem~\ref{mt2} proven above we have that $L(1,\chi_*^{\ell})^{k'}$ and $L(1,(\tfrac \cdot p)\chi_*^{\ell})^{k'}$ are linearly independent over $\Q(\xi_{p-1})$ for all $\ell$ odd. Thus, since $i\in\Q(\xi_{p-1})$, we have $\det (M_{g,1})\neq0$ and so the above system can be written as $\boldsymbol f_1+M_{g,1}^{-1}M_{g,t}\boldsymbol f_t=\boldsymbol 0$. We also observe that by Corollary~\ref{det1} for all $j$, $\sigma_{g^j}(M_{g,1}^{-1}M_{g,t})=M_{g,1}^{-1}C_-(\boldsymbol u_j^-)^{-1}C_-(\boldsymbol u_j^-)M_{g,t}=M_{g,1}^{-1}M_{g,t}$. In particular, $M_{g,1}^{-1}M_{g,t}$ has entries in $K$. Thus, the system $\boldsymbol f_1+M_{g,1}^{-1}M_{g,t}\boldsymbol f_t=0$ is defined over $K$, has rank $\frac{p-1}4$ and has $\frac{p-1}2$ free variables, whence the equality in~\eqref{me} holds in this case.  

We also notice that proceeding as above we also obtain $\det(M_{g,t})\neq0$. In particular, we have that if $\boldsymbol f_1=\boldsymbol0$ or if $\boldsymbol f_t=\boldsymbol0$, then $M_{g,1}\boldsymbol f_1+M_{g,t}\boldsymbol f_t=\boldsymbol 0$ has no non-trivial solutions. Equivalently, there is no solution to $D(1,f)=0$ with $f$ odd and supported only on either square residues or on square non-residues.

\medskip
Now, we prove Theorem~\ref{mt2} in the case $p\equiv 1\mod 4$, $(k,p-1)=2$. As above, it suffices to consider the odd characters case.
We assume $\sumstar_{\chi}a_{\chi}L(1,\chi)^k=0$ with $a_\chi=0$ if $\chi$ is  even and  $a_\chi\in \Q$ if $\chi$ odd. 
By~\eqref{effl} we have 
\es{\label{effl1}
\sumstar_{\chi}a_{\chi}\prbigg{\sumstar_{m\mod p}\chi(m)\cot\pr{\pi \tfrac{ m}p}}^k=0
}
and since we  have $\Q(\xi_p)\cap \Q(\xi_{p-1})=\Q$, then there exists an automorphism $\sigma$ of $\Q(\xi_p,\xi_{p-1})$ which leaves $\Q(\xi_p)$ invariant and send $\xi_{p-1}\mapsto \xi_{p-1}^{1+\frac{p-1}{2}}$. Notice that $(1+\frac{p-1}2,p-1)=1$ and so this automorphism is well defined. Also since every odd character $\chi$ can be written as $\chi(m)=\chi_0(m)\e{\frac{j\nu_g(m)}{p-1}}$ for some $j$ odd and where $g$ is a generator of $(\Z/p\Z)^*$, then $\sigma(\chi(m))=\chi(m)\e{\frac{j\nu_g(m)}{2}}=\chi(m)(\frac mp)$. Thus, applying $\sigma$ to~\eqref{effl1} we obtain
\est{
\sumstar_{\chi}a_{\chi}\prbigg{\sumstar_{m\mod p}(\tfrac mp)\chi(m)\cot\pr{\pi \tfrac{ m}p}}^k=0
}
or, equivalently, $\sumstar_{\chi}a_{\chi}L(1,(\frac\cdot p)\chi)^k=0$. Summing this equation with the one we originally had, we obtain
\est{
\sumstar_{\chi}a_{\chi}\pr{L(1,\chi)^k\pm L(1,(\tfrac\cdot p)\chi)^k}=0
}
or, equivalently $D_k(1,f_{\pm})=0$, where $f_{\pm}=\sum^*_{\chi}a_{\chi}(1\pm(\frac{\cdot}p))\chi$. Notice that $f_+$ and $f_-$ are odd and take values in $\Q(\xi_{p-1})$. Also, $f_{+}$ is supported only on square residues and $f_{-}$ only on square non-residues. Thus, by what proven above we have that $f_{+}$ and $f_-\equiv 0$ are identically zero and thus so is $f\equiv 0$, as desired.

\medskip
The proof of the equality in~\eqref{me} and Theorem~\ref{mt2} in the case $p\equiv 5\mod 8$, $(k,p-1)=4$ is analogous to the case $p\equiv 3 \bmod 4$, $(k,p-1)=2$. Let $k =4k'$ and $t=g^{k'}$. The system~\eqref{bf}  is
equivalent to 
\est{
\sum_{0\le j < (p-1)/2} f(t^{2j})x_k(t^{2j+4\ell};p)
+ \sum_{0\le j < (p-1)/2} f(t^{2j+1})x_k(t^{2j+1+4\ell};p)=0
} 
for $0\le \ell < \frac{p-1}{4}$. Since $-1\equiv  t^{(p-1)/2} \mod p$ and $\frac{p-1}2\equiv 2 \mod 4$, the span of $t^{2j}$ when $j$ runs over $\{0,\ldots, \frac{p-5}{4}\}$ is the same as $\pm t^{4j}$. Thus this system is equivalent to  
 $M_{g,1}'\boldsymbol f_1+M_{g,t}'\boldsymbol f_t=\boldsymbol 0$
 with $\boldsymbol f_r = (f(rt^0),f(rt^{4}),\ldots,
 f(rt^{p-3}))^T$ 
 and we use Corollary~\ref{fcd2b} to compute the determinant of the two matrices. For $r\in\{1,t\}$ we obtain 
$$
\det(M_{g,r}')=\kappa\prod_{\ell=0\atop \ell \text{ odd}}^{(p-1)/2}\pro{L(1,\chi_*^{\ell})^{k''}+i^{\eps(r)}L(1,(\tfrac \cdot p)\chi_*^{\ell})^{k''}}\pro{L(1,\chi_*^{\ell})^{k''}-i^{\eps(r)} L(1,(\tfrac \cdot p)\chi_*^{\ell})^{k''}}
$$
where $\eps(t)=0$, $\eps(1)=1$,  $\kappa\in\R_{\neq0}$ and $k''=k/2$ (so that $(k'',p-1)=2$). We proved above the linear independence over $\Q$  of the $L(1,{\chi})^{k''}$ with $\chi $ odd which implies 
$\det(M_{g,t}')\not =0$. Moreover we have 
\est{
L(1,\chi_*^{\ell})^{k''}\pm iL(1,(\tfrac \cdot p)\chi_*^{\ell})^{k''}
=  \pro{L(1,\chi_*^{\ell})^{k'}- \xi_8^{4\mp1}L(1,(\tfrac \cdot p)\chi_*^{\ell})^{k'}}\pro{L(1,\chi_*^{\ell})^{k'}+
\xi_8^{4\mp1}L(1,(\tfrac \cdot p)\chi_*^{\ell})^{k'}}.
}
We may use Theorem~\ref{mt2} with $(k,p-1)=1$ in the more general case proven above with $K = \Q(\xi_8)$ since it satisfies $K(\xi_{p-1})\cap \Q(\xi_p)=\Q$.  It follows that we also have $\det(M_{g,1}')\not =0$ and 
proceeding as above we obtain the equality in~\eqref{me} and Theorem~\ref{mt2}.
\medskip 

Finally, it remains to prove that the inequality~\eqref{me} always holds. We consider the case $v_2(p-1)> v_2(k)$ only, the other case being analogous. We write $k=k'u$ with $(k,p-1)=u$ and $t=g^{k'}$, obtaining the system
$$M_{g,1}\boldsymbol f_1+M_{g,t}\boldsymbol f_t+\cdots +M_{g,t^{u-1}}\boldsymbol f_{t^{u-1}}=\boldsymbol0
$$
with $\boldsymbol f_{r}=(f(r t^0),f(r t^u),\dots, f(r t^{p-1-u}))^T$. This is a system of $\frac{p-1}{2u}$ equations in $\frac{p-1}2$ variables. Also, the automorphisms $\sigma_{t^j}$ do not change the system since their effect is just that of multiplying all the above matrices by $C_-(\boldsymbol u_j^-)$ on the left. Thus, the system admits a base of solutions in $K$  and~\eqref{me} follows.
\end{proof}

\begin{proof}[Proof of Theorem~\ref{algt}]
Corollaries~\ref{fcd1} and~\ref{fcd1b} give the non-vanishing of the determinant there computed and in the proof of Theorem~\ref{mt} and Theorem~\ref{mt2} we also showed that also the determinants computed in Corollaries~\ref{fcd2} and~\ref{fcd2b} are non-zero. Proposition~\ref{algt} then follows with the same argument as above.
\end{proof}

\bibliographystyle{alpha}
\bibliography{nonvanishing}

\end{document}